\documentclass[a4paper,11pt]{amsart}

\usepackage[utf8]{inputenc}

\usepackage{pgfplots}
\pgfplotsset{compat=newest}

\usepackage{amssymb,mathtools,amsfonts,latexsym,rawfonts, mathrsfs, lscape}
\usepackage{stmaryrd,tikz,pgfplots}
\usepackage{marginnote}  
\usepackage{hyperref}
\usepackage[all]{xy}
\usepackage{amsthm,latexsym,amssymb}
\usepackage{tensor}

\usepackage{tikz-cd} 

\usepackage{multirow}

\usepackage{array, tabularx}

\usepackage{booktabs} 
\usepackage{siunitx}

\usepackage{graphicx}
\usepackage{subfigure}

\usepackage{tikz-cd}
\usepackage{tikz}
\usepackage{color}

\usepackage{dynkin-diagrams}
\usepackage{mathdots}
\usepackage{tabularx}
\usepackage{mathtools}

\usepackage[top=1.5in, bottom=0.9in, left=0.9in, right=0.9in]{geometry}

\newtheorem{theo}{Theorem}[section]
\newtheorem*{theo*}{Theorem}

\newtheorem*{lem*}{Lemma}
\newtheorem{cor}{Corollary}[section]
\newtheorem*{cor*}{Corollary}
\newtheorem{prop}{Proposition}[section]
\newtheorem*{prop*}{Proposition}

\theoremstyle{definition}

\newtheorem*{defi*}{Definition}

\newtheorem*{conj*}{Conjecture}
\theoremstyle{definition}

\newtheorem*{ex*}{Example}

\newenvironment{rmk}%
{\vskip6pt%
\noindent%
{\it Remark.}}%
{\vskip6pt}

{\vskip6pt%
\noindent%
{\it Notation.}}%
{\vskip6pt}

{\vskip6pt%
\noindent%
{\it Observation.}}%
{\vskip6pt}

\theoremstyle{plain}
\newtheorem{thmint}{Theorem}

\newcommand\loceq{\mathrel{\stackrel{\makebox[0pt]{\mbox{\normalfont\tiny loc}}}{=}}}
\renewcommand{\[}{\begin{equation*}}
\renewcommand{\]}{\end{equation*}}

\def\C{\mathbb{C}}

\DeclareMathOperator{\im}{i}
\DeclareMathOperator{\imm}{Im}

\def \N {\mathbb N}

\def \A {\mathcal A}

\def \O {\mathcal O}

\def \H {\mathbb H}
\def \X {\mathcal X}

\def \K {\mathcal K}

\def \Q {\mathbb Q}
\def \R {\mathbb R}

\def \C {\mathbb C}
\def \Z {\mathbb Z}
\def \P {\mathbb P}
\def \B {\mathbb B}

\def \g {\mathfrak g}

\def \Na {\nabla}

\newcommand{\p}{\partial}
\renewcommand{\bar}{\overline}

\title{Calabi--Yau locally conformally K\"ahler manifolds}

\author{Giuseppe Barbaro}
\address{Giuseppe Barbaro\newline
		\textsc{\indent Institut for Matematik, Aarhus University\newline 
			\indent 8000, Aarhus C, Denmark}}
\email{g.barbaro@math.au.dk}

\author{Alexandra Otiman}
\address{Alexandra Otiman \newline
		\textsc{\indent Institut for Matematik, Aarhus University\newline 
			\indent 8000, Aarhus C, Denmark\newline
			\indent \indent and\newline
			\indent Institute of Mathematics ``Simion Stoilow'' of the Romanian Academy\newline 
			\indent 21 Calea Grivitei Street, 010702, Bucharest, Romania}}
	\email{aiotiman@math.au.dk} 
\thanks{Both authors are supported by a DFF Sapere Aude grant ``Conformal geometry: metrics and cohomology". The first author is a member of GNSAGA of INdAM.}

\begin{document}
	
\begin{abstract}
We study compact locally conformally K\"ahler (lcK) manifolds which are Calabi--Yau, in the sense that $c_1^{BC}(X)=0$. 
First of all, we prove that all the known lcK manifolds which are Calabi--Yau are Vaisman.
Then we prove that an lcK Chern--Ricci flat metric that is Gauduchon is necessarily Vaisman. 
Finally, specializing to Calabi--Yau solvmanifolds with left-invariant complex structure, we prove that a left-invariant metric is lcK if and only if it is Vaisman. Therefore, they are finite quotients of the Kodaira manifold.
\end{abstract}
	
\maketitle
	
\section{Introduction}\label{sec: intro}

The notion of Calabi--Yau manifolds has traditionally been introduced and studied in the K\"ahler setting, and has been defined as the vanishing of the first Chern class.
In the non-K\" ahler counterpart, several analogs of the Calabi--Yau notion have been proposed, since simply $c_1(X)=0$ might be a too broad condition. 
We adopt here the point of view introduced by Tosatti in \cite{tos}, which represents a condition of cohomological nature. Namely, a compact complex manifold $X$ is said Calabi--Yau if the first Chern class of its anti-canonical bundle $K_X$ vanishes in the Bott--Chern cohomology $H^{\bullet,\bullet}_{BC}(X)=\tfrac{\ker d}{\imm\p\bar\p}$, which is a refinement of the de Rham cohomology. We write $c_1^{BC}(X)=0$.

In this note we study compact locally conformally K\" ahler manifolds with special first Bott--Chern class, focusing on the Calabi--Yau case. {\em Locally conformally K\" ahler} (lcK) metrics are Hermitian metrics $\omega$ locally admitting a conformal change that makes them K\"ahler. Equivalently, they satisfy the integrability condition $d\omega=\theta \wedge \omega$ with $\theta$ a closed $1$-form called {\em Lee form}.
In this paper, we will call a metric lcK if $\theta$ is not exact, i.e. if the  metric does not admit a global conformal change to a K\"ahler one. These are sometimes called {\em strictly lcK} in the literature, so we just decided to absorb the 'strict' in the definition.
If in addition $\theta$ is also parallel with respect to the Levi--Civita connection of $\omega$, then the metric is called {\em Vaisman} and shares many remarkable properties with K\" ahler metrics. In \cite{i}, Vaisman manifolds for which $c_1^{BC}(X)$ has a sign were studied, including the Calabi--Yau ones, for which a version of the Beauville-Bogomolov decomposition theorem was proven. 
A natural question that arises 
is whether there exists any Calabi--Yau lcK manifold that is not Vaisman. On this regard, we give a sufficient condition on the Chern--Ricci curvature of an lcK metric to be Vaisman. 

\begin{thmint}\label{thm A}
    Let $(M, J)$ be a compact lcK manifold such that the metric $\omega$ is Chern--Ricci flat (or more generally, $Ric(\omega)=tdJ\theta$, $t \leq 0$). If $\omega$ is Gauduchon, then it is Vaisman.
\end{thmint}

The motivations to consider locally conformally K\"ahler metrics that are Chern--Ricci flat, or more generally, their Chern--Ricci form is proportional to $dJ\theta$ are multiple. 
Firstly of all, the class of $dJ\theta$ is a non-zero element in $H^{1, 1}_{BC}(X)$ that represents also the Chern curvature of the holomorphic flat line bundle associated to the Lee form $\theta$, endowed with the Hermitian structure induced from $\omega$. 
Secondly, according to a conjecture of Ornea-Verbitsky \cite{ov25}, $-dJ\theta$ is Bott--Chern cohomologous to a positive current on any lcK manifold, which would lead to a sign of the first Bott--Chern class in this case. 
Moreover, it is well understood that the right Einstein condition to impose for Vaisman metrics is $Ric(\omega)=tdJ\theta$ for $t\in\R$, meaning that we do not ask the Ricci form to be a multiple of the metric but of the transversal K\"ahler metric. On this regard, we provide another motivation for this Einstein equation in Section \ref{s: sasaki} as we relate it to the {\em Sasaki $\eta$-Einstein} equation for the underlying Sasaki structure.

We remark that on a Calabi--Yau manifold, any Hermitian metric is conformal to a Chern--Ricci flat one, but this priviledged metric is not necessarily the Gauduchon metric in the corresponding conformal class.
Indeed, explicit lcK Chern--Ricci flat metrics which are not Vaisman can be produced, see the examples in Section \ref{s: main}.
However, the examples we have of lcK manifolds suggest that even if these conditions do not force Vaisman on the level of the metric, they do it on the level of the manifold.
Let us explain this more carefully.
We prove that all the known lcK non-Vaisman manifolds are not Calabi--Yau. These are some {\em Inoue surfaces, Oeljeklaus–Toma manifold, Hopf manifolds and Kato
manifolds}. We also prove that the same holds for lcK manifolds obtained by deformations of Vaisman ones or modifications of other lcK's, since these are the only know constructing methods.
We state this result here in a contracted version while more details are in Section \ref{s: known lck}.
\begin{thmint}\label{thm B}
    All the known non-Vaisman lcK manifolds are not Calabi--Yau.
\end{thmint}
These examples give evidence towards the possibility that compact lcK manifolds with vanishing first Bott--Chern class are actually Vaisman. 
We thus state this as a conjecture.
\begin{conj*}
    Any compact lcK Calabi--Yau manifold is Vaisman.
\end{conj*}


In the homogeneous setting, the situation becomes much simpler. Indeed, for a compact homogeneous manifold equipped with an invariant Hermitian structure, the metric is automatically Gauduchon and if the Chern--Ricci form is $\p\bar\p$-exact then it is actually zero.
We thus can specialize our analysis of the lcK Calabi--Yau conditions to solvmanifolds, obtaining, in particular, a characterization of the ones with a left-invariant Vaisman structure.
\begin{thmint}\label{thm C}
    Let $(\Gamma \backslash G, J)$ be a solvmanifold with left-invariant complex structure admitting a left-invariant lcK metric. The following are equivalent:
  \begin{enumerate}
\item $(\Gamma \backslash G, J)$ is Calabi--Yau,
      \item $(\Gamma \backslash G, J)$ admits a left-invariant Vaisman metric.
  \end{enumerate}
  Moreover, such a solvmanifold is a finite quotient of a Kodaira manifold.
\end{thmint}
We shall mention that the classification of geometric structures on solvmanifolds is a wide and long-standing problem. For lcK and Vaisman structures on solvmanifolds and nilmanifolds we recall \cite{has06, uga07, saw07, kas, saw17, baz17, ao, gomes}.
In particular, a characterization of Vaisman solvmanifolds with complex structure not necessarily left-invariant was already obtained in \cite{gomes}, while we focus on the left-invariant case since we link the Lie algebra structure with the notion of Calabi--Yau. 
As we shall see in Section \ref{sec: solv}, the class of invariant Vaisman metrics can be also characterized as invariant {\em lcK with potential}. Indeed, while Vaisman metrics are in particular lcK with potential, but not vice-versa in the general context, in the invariant setting, the two notions coincide.

Finally, we remark the difference with the non-integrable setting, where there are solvmanifolds which admit left-invariant locally conformally almost-K\"ahler structures with vanishing Chern--Ricci tensor but $\theta$ not Levi--Civita-parallel. These examples are provided in \cite[4.1 (i) and (ii)]{bl}.

\hfill

\noindent {\it Acknowledgements.} We would like to thank Nicolina Istrati for a careful reading of the preprint and useful remarks.



\section{Preliminaries}\label{sec: prelim}

Any Hermitian metric $\omega$ on a complex manifold $X$ induces a Hermitian structure on $K_X$. Its Chern curvature with respect to this structure is given locally by
$\Theta_{K_X, \omega}=\mathrm{i}\partial \overline{\partial}\, \mathrm{log} \, \mathrm{det}\, \omega$ and for any two Hermitian metrics, $\omega_1$ and $\omega_2$, it holds 
\begin{equation*}
    \Theta_{K_X, \omega_1}-\Theta_{K_X, \omega_2}=\mathrm{i}\partial \overline{\partial}\mathrm{log} \frac{\omega_1^n}{\omega_2^n}.
\end{equation*}
Therefore, the first Bott--Chern class $c_1^{BC}(X)$ is defined as
\begin{equation*}
    c_1^{BC}(X)=c_1^{BC}(K_X^*)=[-\mathrm{i}\partial \overline{\partial}\, \mathrm{log} \, \mathrm{det}\, \omega] \in H^{1, 1}_{BC}(X),
\end{equation*}
where
$$H^{p,q}_{BC}(X)=\frac{\mathrm{Ker} \left( d:\A^{p,q}\rightarrow\A^{p+q+1}\right)}{\mathrm{Im}\left( \partial \overline{\partial}:\A^{p-1,q-1}\rightarrow\A^{p,q}\right)}.$$
The $2$-form that is locally $-\mathrm{i}\partial \overline{\partial}\, \mathrm{log} \, \mathrm{det}\, \omega$ is the Chern--Ricci form of $\omega$, $Ric(\omega)$. 
Notice that a Calabi--Yau manifold, in the sense that $c_1^{BC}(X)=0$, has automatically $c_1(X)=0$.
Moreover, on a Calabi--Yau manifold, any Hermitian metric $\omega$ satisfies $Ric(\omega)=\mathrm{i}\partial \overline{\partial}f$ for a smooth globally defined function $f$. Therefore, the conformal representative $e^{\frac{f}{n}}\omega$ is Chern--Ricci flat. 

The first Bott--Chern class is {\em positive} (resp. {\em negative}) when it contains a representative that is a semi-positive (resp. semi-negative) $(1,1)$-form.
Notice that, if $c_1^{BC}(X)$ contains a strictly positive or negative form then $X$ K\"ahler. 
Finally, these positivity and negativity notions can be relaxed in the sense of currents by saying that $c_1^{BC}(X)$ contains a positive or negative current.  

\medskip

A complex manifold \((M, J)\) is called \textit{locally conformally K\"ahler} (lcK) if there exists a Hermitian metric \(g\) with fundamental two-form satisfying $d\omega = \theta \wedge \omega$, where \(\theta\) is a closed one-form called the \textit{Lee form}. Vaisman manifolds, introduced in \cite{vai} as generalized Hopf manifolds, admit an lcK metric for which the Lee form is parallel with respect to the Levi-Civita connection. In this setting, the vector fields $\theta^{\sharp}$ and $J\theta^{\sharp}$ are real holomorphic, Killing and satisfy $[\theta^{\sharp}, J\theta^{\sharp}]=0$, so the distribution generated by $\theta^{\sharp}$ and $J\theta^{\sharp}$ is integrable and gives rise to a holomorphic foliation, which is independent from the Vaisman metric, hence it is called \textit{canonical foliation}. Vaisman metrics satisfy
$$\|\theta\|^2\omega=\theta\wedge J\theta - dJ\theta,$$
while structures having this form are characterized as {\em lcK with potential}, and were introduced in \cite{ov10}. More precisely, an lcK metric is called lcK with potential if there exists $f \in C^\infty(M, \R)$ so that 
$$\omega = d_\theta Jd_\theta f, \quad \text{where } d_\theta = d-\theta\wedge.$$

Alternatively, the lcK condition can also be characterized by the existence of an open cover $\{U_i\}_i$ of \(M\) together with smooth functions \(f_i : U_i \to \mathbb{R}\), such that on each \(U_i\), $\omega|_{U_i} = e^{f_i} \omega_i$, where each \(\omega_i\) is a K\"ahler form.  Let $\{U_i\}_i$ be a covering of $M$ such that $\theta_{|U_i}=df_i$. The Levi-Civita connections of the local K\" ahler metrics $e^{-f_i}\omega$ glue to a global connection $\nabla^{W}$, known as the {\em Weyl connection}. While $\nabla^{W}$ preserves the local properties of the Levi-Civita connection of a K\" ahler metric, such as being complex $\nabla^{W}J=0$ and being torsion-free, it does not preserve the metric, but its conformal class, satisfying
    $$ \nabla^W \omega = \theta \otimes \omega .$$ 
Moreover, $\nabla^{W}$ is the unique connection with these properties and it is related to the Levi--Civita connection $\nabla^{LC}$ associated to $\omega$ by
    \begin{equation*}\label{weyl_con}
        \nabla^W_xy = \nabla^{LC}_xy -\frac{1}{2}\theta(x)y -\frac{1}{2}\theta(y)x +\frac{1}{2}g(x,y)\theta^\sharp .
    \end{equation*}

Let $\pi: \widetilde{M} \rightarrow M$ be the universal cover of $M$ and $\pi^*\theta=df$. Then $\Omega:=e^{-f}\pi^*\omega$ is a K\" ahler metric on $\widetilde{M}$ and its Levi-Civita connection $\widetilde{\nabla}^{LC}$ satisfies:
\begin{equation*}
\widetilde{\nabla}^{LC}=\pi^*\nabla^{W}.
\end{equation*}
As a consequence, the Ricci tensor of the Weyl connection satisfies $\pi^{*}Ric^{W}=Ric^{\widetilde{\nabla}^{LC}}=Ric^{Ch}_{\Omega}$, the last equality thanks to $\Omega$ being K\" ahler. Therefore, $Ric^{W}$ has all the symmetries of a Ricci tensor of a K\" ahler metric, and thus the Weyl form can be defined, like in the K\"ahler setting, as $Ric^{W, J}:=Ric^{W}(J\cdot, \cdot)$. 
The Ricci form of $\Omega$, given by $Ric^{Ch, J}_{\Omega}:=Ric^{Ch}_{\Omega}(J\cdot, \cdot)$ satisfies:
\begin{equation*}
    Ric^{Ch, J}_{\Omega}=-\im \partial\overline{\partial}\mathrm{log}\, \mathrm{det}\,\Omega=-\im\partial\overline{\partial}\mathrm{log}\, \mathrm{det}\,\left( e^{-f}\pi^*\omega \right)=n\partial\overline{\partial}f+\pi^{*}Ric^{Ch}_{\omega}=\frac{n}{2}\pi^{*}dJ\theta+\pi^{*}Ric(\omega),
\end{equation*}
So finally the following relation between the Ricci--Weyl form and the Chern--Ricci form holds 
\begin{equation}\label{Weyl-form}
    Ric^{W, J}=\frac{n}{2}dJ\theta+Ric(\omega).
\end{equation}

\medskip

We end this preliminary section recalling that for an lcK manifold the Levi--Civita connection satisfies (see e.g. \cite[(3)]{mor17})
$$\nabla_X J = \frac{1}{2}
 (X \wedge J\theta + JX \wedge \theta),\quad \forall X \in TM,$$
where given $\alpha$ a 1-form and $X,Y\in TM$, $X \wedge\alpha$ is the endomorphism of the tangent bundle defined by
$$ (X \wedge\alpha)(Y) := g(X, Y )\alpha^\sharp - \alpha(Y)X.$$
From this we deduce that
\begin{equation*}\label{eq: NaJtheta}
    \nabla J\theta = -(\nabla\theta)(\cdot,J\cdot) - \frac{1}{2}\|\theta\|^2\omega + \frac{1}{2} \theta\wedge J\theta.
\end{equation*}
In particular, $\nabla J\theta = (\nabla J\theta)^{sym,J_-} + (\nabla J\theta)^{skew,J_+}$, and hence $d^*J\theta =0$. The same would also follow by applying the Lie derivative $\mathcal{L}_{J\theta^\sharp}$ to the identity $\omega=g(J\cdot,\cdot)$.

\section{First Bott--Chern class of lcK non-Vaisman manifolds}\label{s: known lck}

In this section, we go through all the known classes of lcK manifolds and compute their first Bott--Chern class. 
Other than Inoue, Hopf, and Kato surfaces, the known classes of manifolds carrying lcK metrics in higher dimensions are Oeljeklaus--Toma (OT) manifolds, Hopf manifolds and Kato manifolds. 
It is well known that OT and Hopf manifolds have vanishing first Chern class. We shall see that the refined Bott--Chern class is however non-zero. 
For the Kato manifolds, we prove that the first Chern class is not vanishing and a more precise characterization of the first Bott--Chern class is possible when restricting to \textit{toric Kato manifolds}. 
We also prove the non-vanishing of the first Bott--Chern class of blown-up manifolds and modifications in general as this is a constructing method for lcK manifolds.
Finally, we prove the same for lcK with potential non-Vaisman manifolds.
We gather these results in the following statement, whose proof is divided family by family in dedicated subsections. 
\begin{theo*}[Theorem \ref{thm B}]
    All the known lcK non-Vaisman manifolds have non-vanishing first Bott--Chern class. In detail, if $X$ is a
    \begin{itemize}
        \item lcK Inoue surface of type $\mathcal{S}^+$, then $c_1^{BC}(X)<0$.
        \item OT manifold of type $(s,1)$, then $c_1^{BC}(X)<0$.
        \item primary Hopf manifold, then $c_1^{BC}(X) \neq 0$ and it contains a positive $(1,1)$-current. If it is diagonal, then $c_1^{BC}(X) > 0$.
        \item Kato manifold, then $c_1(X) \neq 0$. In particular, $c_1^{BC}(X)\neq0$, and if $X$ is a toric Kato manifold, $c_1^{BC}(X)$ contains a positive $(1,1)$-current.
        \item compact lcK with potential non-Vaisman manifold with $\dim_\C X\geq3$, then $c_1^{BC}(X)\neq0$, and the statement is also true for surfaces if the GSS conjecture holds.
        \item modification of a smooth manifold, then $c_1^{BC}(X)\neq0$.
    \end{itemize}
\end{theo*}

More details are in the statements of Proposition \ref{prop: Inoue} for Inoue surfaces; Proposition \ref{prop: OT} for OT manifolds; Propositions \ref{prop: c1 diag Hopf} and \ref{prop: deformation} for Hopf manifolds; Propositions \ref{p: Kato} and \ref{p: toric Kato} for Kato manifolds; Propositions \ref{p: with potential} and \ref{p: surf with potential} for lcK manifolds with potential; Propositions \ref{prop: blow-up} and \ref{prop: modifications} for modifications.

\subsection{Inoue surfaces}
In \cite{ino1,ino2,ino3} Inoue introduced three families of compact complex surfaces, $\mathcal{S}_M$, $\mathcal{S}^+$, and $\mathcal{S}^-$, as quotients of $\C\times\H$ (where $\H$ denotes the complex upper half-plane) by a cocompact group of affine transformations.
It was proved in \cite{tri} that Inoue surfaces of class $\mathcal{S}_M$ and $\mathcal S^-$, and some wide subclasses of the Inoue surface $\mathcal{S}^+$ admit lcK metrics. 
The Inoue--Bombieri surfaces $\mathcal{S}_M$ have been generalized as Oeljeklaus--Toma manifolds and will be hence treated in Section \ref{s: OT}.
On the other hand, it is known that the surfaces of class $\mathcal{S}^-$ are double covered by surfaces $\mathcal{S}^+$. Therefore, since the properties of the Chern class we want to study are preserved by finite quotients, we will focus on Inoue surfaces of class $\mathcal{S}^+$.
Let us briefly recall their construction.

Given $N=(n_{ij}) \in \mathrm{SL}_2(\Z)$ a  matrix with real eigenvalues $\alpha >1$ and $1/\alpha$ and $(a_1, a_2)^{t}, (b_1, b_2)^{t}$  real eigenvectors corresponding to $\alpha$ and $1/\alpha$. 
Fix some integers $p, q, r$ with $r \neq 0$ and a complex number $\xi\in\C$. 
Let $e_1, e_2$ be defined as
$$e_i = \tfrac{1}{2}n_{i1}(n_{i1}-1)a_1b_1 + \tfrac{1}{2}n_{i2}(n_{i2}-1)a_2b_2 + n_{i1}n_{i2}b_1a_2$$ 
and $c_1, c_2$ defined by
$$(c_1, c_2)=(c_1, c_2) \cdot N^t + (e_1, e_2) + \frac{b_1a_2-b_2a_1}{r}(p, q).$$
Take coordinates $\left<z,w\right>$ on $\C\times\H$, and define the group $G^+_{N, p, q, r, \xi}$ generated by
\begin{equation*}
\begin{cases}
    g_0(z, w) = (z + \xi, \alpha w)\\
    g_i(z, w)  =(z+b_iw+c_i, w+a_i), \qquad \text{for } i=1, 2\\
    g_3(z, w)  =\left(z+ \frac{1}{r}(b_1a_2-b_2a_1), w\right).
\end{cases}
\end{equation*}
Then the Inoue surface $\mathcal{S}^+$ with parameters $(N, p, q, r, \xi)$ is defined as $\C \times \H/G^+_{N, p, q, r, \xi}$.
These carry lcK metrics if and only if $\xi\in\R$ \cite{tri, bel00}, in which case an lcK metric is given by
$$\omega = \im \frac{1+(\imm z)^2}{(\imm w)^2}dw \wedge d\bar w + \im \frac{\imm z}{\imm w}(dz \wedge d\bar w + dw \wedge d\bar z) + \im dz \wedge d\bar z.$$
We can thus directly compute the first Bott--Chern class of these surfaces and prove that they are not Calabi--Yau.

\begin{prop}\label{prop: Inoue}
    Let $X$ be an lcK Inoue surface of type $\mathcal{S}^+$, then $c_1^{BC}(X)<0$.
\end{prop}
\begin{proof}
    With the same notations above, we can compute the Chern--Ricci form of $\omega$.
    This is locally given by
    $$Ric(\omega) \loceq \im \p\bar\p\log\left(\imm w\right)^2,$$
    and globally by
    $$Ric(\omega)= -\frac{\im}{2}\frac{dw\wedge d\bar w}{(\imm w_i)^2}. $$
    This is a semi-negative form, hence, in particular, it is not $\p\bar\p$-exact.
\end{proof}

\subsection{Oeljeklaus--Toma manifolds}\label{s: OT}
Oeljeklaus--Toma manifolds are complex manifolds introduced in \cite{ot05} with a number-theoretical construction to generalize Inoue--Bombieri surface of type $S_M$. 
They are compact quotients of $\H^s\times\C^t$, hence we will refer to them as OT-manifolds of type $(s,t)$.
More precisely, their construction is as follows (see \cite{ot05} for details).
Fix $s,t\in\N\setminus\{0\}$, and take an algebraic number field given as $K\simeq\Q[X]\slash(f)$,
where $f\in\Q[X]$ is a monic irreducible polynomial with $s$ real roots, and $2t$ complex roots.
The field $K$ admits $s+2t$ embeddings in $\C$ by the conjugate roots, more precisely, $s$ real embeddings and $2t$ complex embeddings:
$$
\begin{gathered}
\sigma_1,\ldots,\sigma_s\colon K\to \mathbb{R},\\
\sigma_{s+1},\ldots,\sigma_{s+t},\sigma_{s+t+1}=\overline\sigma_{s+1},\ldots,\sigma_{s+2t}=\overline\sigma_{s+t} \colon K\to\mathbb{C}.
\end{gathered}
$$
Let's call $\mathcal O_K$ the ring of algebraic integers of $K$; $\mathcal O_K^*$ the multiplicative group of units of $\mathcal O_K$; and $\mathcal O_K^{*,+}$ the group of totally positive units.
Thanks to \cite{ot05}, there always exists an {\em admissible subgroup} $U\subset\O_K^{*,+}$, of rank $s$, such that the fixed-point-free action
$$ \O_K\rtimes U\circlearrowleft\mathbb H^s\times\mathbb C^t $$
induced by
\begin{align*}
    T\colon &\O_K \circlearrowleft \H^s\times\C^t, \\
T_a(w_1,\ldots,w_s,z_{s+1},\ldots,z_{s+t}) &:= (w_1+\sigma_1(a),\ldots,z_{s+t}+\sigma_{s+t}(a)),\\
R\colon &\O_K^{*,+} \circlearrowleft \H^s\times\C^t, \\
R_u(w_1,\ldots,w_s,z_{s+1},\ldots,z_{s+t}) &:= (w_1\cdot \sigma_1(u),\ldots,z_{s+t}\cdot \sigma_{s+t}(u)),
\end{align*}
is properly discontinuous and co-compact.
Thus the {\em Oeljeklaus--Toma manifold} of type $(s,t)$ associated to the algebraic number field $K$ and to the admissible subgroup $U$ is
$$ X(K,U) :=\H^s\times\C^t \slash \mathcal O_K\rtimes U.$$

It was proven in \cite{kas} that OT manifolds can be endowed with
a structure of solvmanifold, in such a way that the complex structure is left-invariant.
Moreover, they admit an lcK structure if and only if $t=1$ \cite{dv23}, and this structure is not Vaisman.
So let us restrict to this case.

\begin{prop}\label{prop: OT}
    Let $X$ be an OT manifold of type $(s,1)$, then $c_1^{BC}(X)<0$.
\end{prop}
\begin{proof}
    Consider coordinates $\left<w_1,\ldots,w_s\right>$ on $\H^s$ and $z$ on $\C$.
    We have the standard Hermitian metric on $X$ given by
    $$\omega= \sum_{i=1}^s\frac{\im}{(\imm w_i)^2}dw_i\wedge d\bar w_i + \im\left(\prod_{i=1}^s \imm w_i\right)dz\wedge d\bar z.$$
    The Chern--Ricci form is then locally given by
    $$Ric(\omega) \loceq \im \p\bar\p\log\left(\prod_{i=1}^s\imm w_i\right),$$
    and globally by
    $$Ric(\omega)= -\frac{\im}{4}\sum_{i=1}^s\frac{dw_i\wedge d\bar w_i}{(\imm w_i)^2}, $$
    Therefore, it is semi-negative and not $\p\bar\p$-exact. 
\end{proof}

\subsection{Hopf manifolds}\label{sec: Hopf}

Primary Hopf manifolds are complex manifolds constructed as quotients of $\mathbb{C}^n \setminus \{0\}$ by an infinite cyclic group, whereas secondary Hopf manifolds are finite quotients of them.
We thus focus on primary Hopf surfaces. 
They are known to carry lcK metrics by \cite{go98, ov16, ov23, ko}, while they are Vaisman only when their fundamental group is generated by a diagonal contraction $\gamma:\mathbb{C}^n \setminus \{0\} \rightarrow \mathbb{C}^n \setminus \{0\}$, $\gamma(z_1, \ldots, z_n)=(\lambda_1 z_1, \ldots, \lambda_n z_n)$, see e.g. \cite{io25,bel00}. 
In this case, when also $|\lambda_1|=\ldots=|\lambda_n|$, it was proven in \cite{tos} that they cannot be Calabi--Yau, by using an explicit Vaisman metric. 
For a generic Hopf manifold, an explicit lcK metric cannot be given. 
Therefore, we provide a more general argument that does not rely on explicit computations starting from an lcK metric. 
This applies to any primary Hopf manifold, regardless of the associated contraction, while we remark that for Hopf surfaces it was noticed in \cite{tw} that $c_1^{BC}\neq0$. 
It is based on a deformation result that allows to deform any primary Hopf manifold in a diagonal one. 
Therefore, we will first prove that $c_1^{BC}>0$ for diagonal (i.e. Vaisman) Hopf manifolds and then extend this to any Hopf manifold.

\begin{prop}\label{prop: c1 diag Hopf}
    Let $X = \big(\mathbb C^n \setminus \{0\}\big)/\langle d_\lambda\rangle$ be a primary Hopf manifold generated by a diagonal holomorphic contraction $d_\lambda$. Then $c_1^{BC}(X) > 0$.
\end{prop}
\begin{proof}
    Let us consider the meromorphic $(n,0)$-form on $\C^n\setminus\{0\}$ given by
    $$\sigma=\frac{1}{z_1\cdots z_n}dz_1\wedge\cdots\wedge dz_n.$$
    This is $d_\lambda$-invariant, hence it induces a meromorphic section of the canonical bundle $\K_X$ of $X$, which we still call $\sigma$.
    In other words, the canonical is the line bundle associated to the divisor generated by zeros and poles of $\sigma$. We write this as
    $$\K_X = [(\sigma)], \quad \text{where }\quad (\sigma)=-\sum_{i=1}^n V_i, \quad \text{with }\quad V_i=\{z_i=0\}/\langle d_\lambda\rangle. $$
    Then the Chern class of the canonical is 
    $$c_1(\K_X) = -\sum_{i=1}^n\eta_{V_i}, $$
    where $\eta_{V_i}$ is the Poincaré dual of the fundamental class associated to the analytic subvariety $V_i\subset X$, see e.g. \cite[Proposition pg. 141]{GriHar}.
    Therefore, we consider a Gauduchon metric $\omega$ and integrating against it we get
    \begin{equation*}
        \int_X c_1^{BC}(X)\,\omega^{n-1} =\sum_{i=1}^n\int_{V_i}\omega^{n-1} > 0. 
    \end{equation*}
    This proves that $c_1^{BC}(X)$ cannot be zero, otherwise the integral would vanish by Stokes' theorem since $\p\bar\p\omega^{n-1}=0$.
    Therefore, it is a nonzero element in $H^{1, 1}_{BC}$, which, by \cite{io25}, is one-dimensional for any Hopf manifold and generated by $[dJ\theta]_{BC}$ ($\theta$ being the Lee form of an lcK metric). 
    In particular, we can choose $\theta$ to be the Lee form of a Vaisman metric on $X$. Then, it is known that $-dJ\theta$ is a transversal K\"ahler form, hence it is semipositive.
    Using again the positivity of the above integral, it follows that $a$ must be negative and $c_1^{BC}(X)>0$.
\end{proof}

We can now extend this positivity to any Hopf manifold by deformations. In this case, even if we cannot guaranty the existence of a semi-positive form, we can still guaranty the integral of $c_1^{BC}$ to stay positive. Therefore, for a general Hopf manifold, $c_1^{BC}$ will be positive in the sense of currents.

\begin{prop}\label{prop: deformation}
    Let $X = \big(\mathbb C^n \setminus \{0\}\big)/\langle \gamma\rangle$ be a primary Hopf manifold generated by a holomorphic contraction $\gamma$. Then $c_1^{BC}(X) = [adJ\theta]_{BC}$ with $a<0$. In particular, $c_1^{BC}(X)\neq0$, and contains a positive $(1,1)$-current.
\end{prop}    
\begin{proof}
    It is known that $\gamma$ can be assumed to be a $\lambda$-resonant biholomorphic polynomial map whose derivative $\gamma'(0)$ at zero has eigenvalues
    $$\lambda = (\lambda_1,\ldots,\lambda_n) \in \C^n, \quad \text{with }\quad 0\leq |\lambda_1| \leq \cdots \leq |\lambda_n| < 1.$$
    Let's call $d_\lambda$ the diagonal contraction given by the diagonal part of $\gamma$.
    Then it is known, see e.g. \cite[pg. 242]{Hae}, that there exits a holomorphic family of Hopf manifolds $\pi: \X \rightarrow \C$ such that $\X_0:=\pi^{-1}(0) = X_{d_\lambda}$ and for any $t\in\C^*$ we have $\X_t:=\pi^{-1}(t)\cong X$, where $X_{d_\lambda}$ is the diagonal Hopf manifold obtained as $\big(\mathbb C^n \setminus \{0\}\big)/\langle d_\lambda\rangle$.
    Being a holomorphic family, it comes equipped with a family of Hermitian metrics $\omega_t$ on each $\X_t$ varying smoothly on $t\in\C$.
    Moreover, the Gauduchon representative in any Hermitian class is the solution of an elliptic problem which depends smoothly on the initial data. 
    Therefore, we can suppose the $\omega_t$ to be Gauduchon.
    Now we just look at
    $$S(t)=\int_{\X_t}c_1^{BC}(\X_t)\,\omega_t^{n-1} = \int_{\X_t} s^{Ch}(\omega_t)\,\omega_t^n,$$
    where $s^{Ch}$ is the Chern scalar curvature, and the integrals are actually done all on the same smooth manifold since the $\X_t$ are all diffeomorphic.
    The function $S$ is clearly varying smoothly on $t\in\C$ and it is positive for $t=0$ thanks to Proposition \ref{prop: c1 diag Hopf}.
    Therefore, for $|t|$ small enough
    $$\int_{\X_t}c_1^{BC}(\X_t)\,\omega_t^{n-1} > 0,$$
    and hence, since $\omega_t$ is Gauduchon, $c_1^{BC}(X)\cong c_1^{BC}(\X_t)\neq0$.

    We now argue as in the proof of Proposition \ref{prop: c1 diag Hopf}, using \cite{io25} to guarantee that $c_1^{BC}(X)=[adJ\theta]_{BC}$ for some $a \neq 0$.
    It was proved in \cite{ov25} that $-dJ\theta$ is Bott--Chern cohomologous to a positive current. Therefore, $a$ has to be negative and $c_1^{BC}(X)$ contains a positive current.
\end{proof}

\subsection{Blow-ups and modifications}\label{sec: blow-up}
Here we study the first Bott-Chern class of smooth modification of complex manifolds, where we recall that a modification of a complex manifold $X$ is a birational morphism $\pi:Y\rightarrow X$ of complex manifolds which is not an isomorphism.
Thanks to \cite{ovv}, the blow-up of an lcK manifold along a complex submanifold also admits an LCK structure when this submanifold is of {\em induced globally conformally K\"ahler type}. 
Moreover, there also are some more general modifications of lcK manifolds which are not blow-ups and admit lcK structures. We provide here an example.

\begin{ex*}
    Suppose that $\pi:\hat\B\rightarrow\B$ is a K\"ahler modification in $0$ of the unit ball $\B\subset\C^n$. That is a modification of $\B$ such that $\hat\B$ admits a K\"ahler metric $\omega_K$. 
    Now, for any lcK manifold $(X,\omega)$, we indicate with the same symbol $\pi$ the associated modification $\pi:\hat X\rightarrow X$ of $X$ in $p\in X$ obtained by embedding $(\B,0)\hookrightarrow(X,p)$.
    We want to show that $\hat X$ is lcK.
    Since $\omega$ is lcK it satisfies $d\omega=\theta\wedge\omega$, and (up to a conformal change) we can take $\theta_{|_\B}=0$.
    Suppose now that we have a closed $(1,1)$-form $\Omega$ on $\hat\B\hookrightarrow\hat X$ which is K\"ahler in a neighborhood of the exceptional locus and whose support is strictly contained in $\hat\B$.
    Then, we can construct an lcK metric on $\hat X$, $\hat\omega=\pi^*(\omega)+\varepsilon\,\Omega$ for $\varepsilon>0$ small enough. 
    Indeed,
    $$d\hat\omega = d\pi^*(\omega) = \pi^*(\theta)\wedge\pi^*(\omega)= \pi^*(\theta)\wedge\hat\omega, $$
    where the last equality follows from the fact that $\pi^*(\theta)\wedge\Omega=0$ since $\Omega$ is supported in $\hat\B$ and $\pi^*(\theta)$ outside.
    The existence of $\Omega$ with the required properties follows by adapting the argument of \cite[Lemma 10.2]{iopr}. Indeed, we need to consider a function $\rho=\rho(\|z\|^2)$ on $\B$ which is strictly plurisubharmonic on $\B_{r_2}$ and zero outside $\B_{r_3}$, for $\B_r$ the ball in $0$ of radious $r$, and $0<r_1<r_2<r_3<1$. If we denote by $\hat\B_r$ the $\pi$-modification of $\B_r$, then there exists $\Omega$ a closed $(1,1)$-form on $\hat\B$ such that
    $$\Omega_{|_{\B_{r_1}}}=\omega_K, \quad \text{and }\quad \Omega_{|_{\B\setminus \B_{r_2}}}=\lambda\pi^*dd^c\rho,$$
    for some constant $\lambda$.
    In particular, $\Omega=0$ outside $\B_{r_3}$, and hence it can be extended to the whole $\hat X$. 
    \hfill$\blacksquare$
\end{ex*}

\begin{rmk}
    The question of whether in general a modification of an lcK manifold admits an lcK structure is very wild. Indeed, following the same argument used in \cite[Theorem 2.8]{ovv} for blow-ups, the problem reduced to understand whether a modification of a K\"ahler manifold is still K\"ahler. This is clearly true for blow-ups, but is not in general and deeply depends on the specific modification.
\end{rmk}

\medskip

Given a modification $\pi:X\rightarrow Y$, the exceptional locus of $\pi$ in $Y$ is a subvariety $E\subset Y$ of codimension one, which is contracted by $\pi$ to a subvariety $V\subset X$ of codimension at least two, see e.g \cite[Corollary 2.5]{shaf}. 
Since $\pi$ is a biholomorphism outside the exceptional loci, the line bundle $\K_Y - \pi^*\K_X $ is supported on $E$, and hence associated with a divisor $\sum_i a_iE_i$, for $E_i$ the irreducible components of $E$ and $a_i\in\R$.
That is,
\begin{equation}\label{eq: canonical modification}\tag{K$_{\rm mod}$}
    \K_Y = \pi^*\K_X + \sum_ia_i[E_i] ,\quad \text{for some } a_i\in\R.
\end{equation}
In general, the coefficients $a_i$ depend on the modification. However, since $X$ and $Y$ are smooth one can argue that they are positive, looking at $\pi$ as a resolution of singularities. A more direct way to see this is by pulling back a local volume form on $X$ around $p\in V$, which will vanish on $E$ but not have poles. 
It follows that the first Chern classes are related by
\begin{equation*}\label{eq: Chern class modification}
    c_1(Y) = \pi^*\left(c_1(X)\right) -\sum_ia_i\,\eta_{E_i}, \quad\text{for some } a_i>0,
\end{equation*}
where $\eta_{E_i}$ is the Poincaré dual of $E_i$.
If $\pi$ is the blow-up of a submanifold $V\subset X$, the coefficients $a_i$ can be computed, giving
$$\K_Y = \pi^*\K_X + (n-v-1)[E], \quad \text{and }\quad c_1(Y) = \pi^*\left(c_1(X)\right) -(n-v-1)\,\eta_{E},$$
where $v$ is the dimension of $V$, and $E$ is the exceptional divisor, see e.g. \cite[pg. 608]{GriHar}.
With this in our hands, we are ready to prove that the first Chern class does not vanish.

\begin{prop}\label{prop: blow-up}
    Let $X$ be a complex manifold, and consider its blow-up $\pi:Y\rightarrow X$ along a complex submanifold $V\subset X$ of codimension $d>1$. Then $c_1(Y)\neq0$. In particular, $c_1^{BC}(Y)\neq0$.
\end{prop}
\begin{proof}
    It is known that the exceptional divisor $E=\pi^{-1}(V)$ is a fiber bundle over $V$ with fiber $\P_\C^{d-1}$. Hence, given a point $p\in V$, we have $F:=\pi^{-1}(p)\cong \P_\C^{d-1}$.
    Moreover, the restriction of the Chern class to $F$ is
    $$c_1(Y)_{|_F} = (\pi^*(c_1(X))_{|_F} - (d-1) (\eta_E)_{|_F} = -(d-1)\, c_1\big(\mathcal{O}_{\P_\C^{d-1}}(-1)\big)_{|_F} \neq0. $$
    Therefore, $c_1(Y)$ is not zero.
\end{proof}

For modifications in general, we have to look at the first Chern class in the Bott--Chern cohomology since it is more precise.

\begin{prop}\label{prop: modifications}
    Let $X$ be a complex manifold, and consider a modification $\pi:Y\rightarrow X$. Then $c_1^{BC}(Y)\neq0$.
\end{prop}
\begin{proof}
    
    Let us first assume that $c_1^{BC}(X)\neq0$. Then for any Gauduchon metric $\omega$ on $X$, the integral of $c_1^{BC}(X)$ against $\omega^{n-1}$ is not zero.
    We thus consider $\tilde\omega:=\pi^*\omega$ and compute
    \begin{align*}
        \int_Y c_1^{BC}(Y)\wedge\tilde\omega^{n-1} &= \int_Y \big[\pi^*\left(c_1^{BC}(X)\right) -\sum_ia_i\,\eta_{E_i}\big]\wedge\tilde\omega^{n-1} \\
            &= \int_Y \pi^*\left(c_1^{BC}(X)\wedge\omega^{n-1} \right) - \sum_i a_i\,\int_{E_i} (\pi^*\omega)^{n-1}\\
            &= \int_X c_1^{BC}(X)\wedge\omega^{n-1} \neq 0.
    \end{align*}
    Here we used the fact that the second integral on the second line is vanishing since $\pi$ contracts $E$ into something of codimension greater than one.
    Since $\tilde\omega^{n-1}$ is a $\p\bar\p$-closed $(n-1,n-1)$-form on $Y$, $c_1^{BC}(Y)$ has to be different from zero in Bott--Chern cohomology.

    Suppose now that $c_1^{BC}(X)=0$, and consider any Gauduchon metric $\check\omega$ on $Y$. It holds
    \begin{align*}
        \int_Y c_1^{BC}(Y)\wedge\check\omega^{n-1} &= \int_Y \big[\pi^*\left(c_1^{BC}(X)\right) -\sum_ia_i\,\eta_{E_i}\big]\wedge\check\omega^{n-1} \\
            &=  - \sum_i a_i \int_{E_i} \check\omega^{n-1} < 0.
    \end{align*}
    Therefore, we conclude again that $c_1^{BC}(Y)$ has to be different from zero in Bott--Chern cohomology.
\end{proof}


\subsection{Kato manifolds}\label{sec: Kato}
In \cite{kat77} Kato introduced a construction method for compact complex manifolds of non-K\"ahler type containing a {\em global spherical shell}, known as Kato manifolds.
Let $\pi: \hat{\B} \rightarrow \B$ be a modification of the standard open unit ball $\B\subset\C^n$ at finitely many points and let $\sigma: \overline{\B} \hookrightarrow \hat{\B}$ be a holomorphic embedding.
Glue small neighborhoods of the two boundary components of $\hat{\B} \setminus \sigma(\B)$ via the local biholomorphism $\sigma \circ \pi$. The resulting manifold, indicated as $X(\pi,\sigma)$, is the wanted compact complex manifold, and the couple $(\pi, \sigma)$ is referred to as a {\it Kato data}. 
Notice that if $\sigma(0)\notin E$ for $E$ the exceptional locus of $\pi$, then $X(\pi,\sigma)$ is just a modification of a Hopf manifold. In addition, Kato manifolds admit small deformations in modification of Hopf manifolds, where the modification of these is still, in some way, $\pi$, \cite{kat77}. More precisely, this modification of the Hopf manifold is obtained removing an open set isomorphic to $\B$ and gluing $\hat\B$ instead. In the following we will use the same symbol $\pi$ for both the above modifications as they are essentially the same.

The existence of lcK metrics on Kato manifolds was characterized in \cite{iopr} in terms of the modification $\pi$. Namely, the Kato manifold $X(\pi,\sigma)$ admits an lcK metric if and only if $\hat{\B}$ admits a K\"ahler metric. This condition is of course verified when $\pi$ is the blow-up of the ball, but there are also other less trivial K\"ahler modifications of $\B$. In any case, we prove that the first Chern class of Kato manifolds is non-vanishing independently on whether they admit an lcK metric or not. For Kato surfaces, one could argue they are not Calabi--Yau by \cite[Theorem 1.4]{tos}, since their canonical bundle is not torsion.

\begin{prop}\label{p: Kato}
    Let $X(\pi,\sigma)$ be a Kato manifold, then $c_1(X)\neq0$. In particular, $c_1^{BC}(X)\neq0$. 
\end{prop}
\begin{proof}
    First of all, let's consider the deformation of $X$ into a modified Hopf manifold $\hat H$, and we keep calling $\pi$ the modification $\pi:\hat H\rightarrow H$. 
    We call $E$ the exceptional locus of $\pi$, with irreducible components $E_1,\dots,E_s$.
    Since the first Chern class is constant in de Rham cohomology under small deformations, it suffices to prove that $c_1(\hat H)\neq0$. 
    Being a modification of a K\"ahler manifold, $\hat\B$ always admits a balanced metric $\omega$ \cite{ab91}.
    Now, suppose that $c_1(\hat H)=0$. Then, it has to be zero also when restricted to $\hat\B \hookrightarrow \hat H$.
    Hence, since $c_1(H)$ vanishes on the contractible set $\B$, we have
    $$0=c_1(\hat H)_{|_{\hat\B}} = \big( \pi^*(c_1(H) - \sum_ia_i\,\eta_{E_i} \big)_{|_{\hat\B}} = \pi^*\left( c_1(H)_{|_{\B}} \right) -\big(\sum_ia_i\,\eta_{E_i} \big)_{|_{\hat\B}} = -\big(\sum_ia_i\,\eta_{E_i} \big)_{|_{\hat\B}} .$$
    However, this is not possible. Indeed, the Poincaré classes $\eta_{E_i}$ are supported in neighborhoods of the $E_i$ as small as we need. Therefore, integrating on $\hat\B$ we have
    $$\int_{\hat\B} \big(\sum_ia_i\,\eta_{E_i}\big)\wedge\omega^{n-1} = \sum_ia_i\int_{E_i}\omega^{n-1} > 0 , $$
    giving the contradiction we needed to conclude the proof.
\end{proof}

A distinguished subclass of Kato manifolds that is related to toric geometry was introduced in \cite{iopr}. These are known as toric Kato manifolds and are characterized by having as construction data a toric modification of $\mathbb{C}^n$ at 0, $\pi: \hat{\mathbb{C}}^n \rightarrow \mathbb{C}^n$, and an equivariant embedding of the ball $\sigma: \mathbb{B} \rightarrow \hat{\mathbb{C}}^n$.
For these manifolds we can further describe the first Bott--Chern class proving that it contains a positive $(1,1)$-current.
To do it, we use the fact that part of the geometry of toric Kato manifolds is expressed in terms of an associated matrix $A \in GL(n, \mathbb{Z})$.

\begin{prop}\label{p: toric Kato}
    Let $X$ be a toric Kato manifold, then $c_1^{BC}(X)$ contains a positive $(1,1)$-current.
\end{prop}

\begin{proof}
    By \cite[Prop. 5.7]{iopr}, the canonical bundle $K_X$ is given by 
    \begin{equation*}
        K_X \otimes L_{\mathrm{det}\, A}\simeq \mathcal{O}_X(-D_T),
    \end{equation*}
    where $D_T$ is the toric divisor associated to $X$ and $L_{\mathrm{det} A}$ is the flat line bundle associated to the representation $\rho: \mathbb{Z} \rightarrow \mathbb{C}^*$, $\rho(1)=(-1)^{\mathrm{det}\, A}$. This gives furthermore $K_X^{-1}\simeq \mathcal O(D_T)\otimes L_{\mathrm{det}\, A}$. Since $L_{\mathrm{det}\, A}$ is flat and $\mathrm{Im} \rho \subset S^1$, $c_1^{BC}(L_{\mathrm{det} A})=0$ and
    \begin{equation*}
        c^{BC}_1(X)=c^{BC}_1(K_X^{-1})=c^{BC}_1(\mathcal O(D_T))=[D_T]_{BC},
    \end{equation*}
    where $[D_T]_{BC}$ is the integration current, a closed positive $(1,1)$--current. 
    
\end{proof}

\subsection{lcK with potential}

A result by \cite{ov10} states that the class of lcK with potential manifolds is closed under small deformations. It follows that deformations of a Vaisman manifold always admit lcK metrics with potential, and so we study whether they are Calabi--Yau or not.
We use also in this case the continuity property of the first Chern class together with a deformation result for lcK manifolds with potential. As a matter of fact, it was proved in \cite{ov10b} that for any lcK manifold with potential of complex dimension at least $3$, there exists a small deformation of which admits a Vaisman
metric.

\begin{prop}\label{p: with potential}
    Let $X$ be a compact lcK manifold with potential with $\dim_\C X\geq3$. Then, if $c_1^{BC}(X)=0$ it is Vaisman.
    In other words, compact lcK with potential non-Vaisman manifold with $\dim_\C X\geq3$ have $c_1^{BC}(X)\neq0$.
\end{prop}
\begin{proof}
    The proof of \cite[Theorem 2.1]{ov10b} shows that $X$ comes from a holomorphic family $\pi:\X\rightarrow\C$ such that $\X_0:=\pi^{-1}(0)$ is Vaisman, and the generic fiber $\X_t:=\pi^{-1}(t)$ for $t\in\C^*$ is isomorphic to $X$.
    Indeed, $X$ can be embedded into a Hopf manifold $H$, and the authors proved that, as a smooth manifold, it can be embedded into a holomorphic family of Hopf manifolds whose central fiber is diagonal (hence Vaisman) and general fiber is isomorphic to $H$.
    Then, if $c_1^{BC}(X)=0$, also $c_1(X)=0$, and this is constant under deformations, hence $c_1(\X_0)=0$. 
    Furthermore, by the same continuity argument of the proof of Proposition \ref{prop: deformation} we get
    $$\int_{\X_0}c_1^{BC}(\X_0)\,\omega_0^{n-1} = 0,$$
    for some Gauduchon metric $\omega_0$ on $\X_0$.
    However, for Vaisman manifolds with vanishing first Chern class, the first Bott--Chern class has a well defined sign, see e.g. \cite{i}, meaning that it has to be either positive, negative, or zero. Thus $c_1^{BC}(\X_0)= 0$.
    Now, thanks to \cite[Theorem E]{i}, any small deformation of a compact Vaisman manifold with vanishing first Bott--Chern class has to be Vaisman.
\end{proof}

The deformation result \cite[Theorem 2.1]{ov10b} depends on the possibility of embed an lcK manifold with potential into a Hopf manifold.
At the present state of art, this can be proved for complex surfaces only using Kodaira's classification and admitting that the {\em Global Shperical Shell} (GSS) conjecture holds true. However, in this case we could give a direct argument as follows.
We first observe that thanks to Proposition \ref{prop: blow-up}, it is enough to look at the first Bott--Chern class of minimal surfaces.
If we suppose that the GSS conjecture is true, then the only minimal class VII surfaces admitting an lcK metric with potential are Hopf surfaces, which we studied in Section \ref{sec: Hopf}. 
On the other hand, the only non-K\"ahler minimal surfaces which are not of class VII are {\em non-K\"ahler elliptic surfaces}, and these are Vaisman, see e.g. \cite{ov10b}. 
We thus state the following result based on the veridicity of the GSS conjecture.

\begin{prop}\label{p: surf with potential}
    Provided that the GSS conjecture holds true, compact lcK with potential non-Vaisman surfaces have $c_1^{BC}(X)\neq0$.
\end{prop}

\section{Calabi--Yau locally conformally K\" ahler manifolds}\label{s: main}

We focus in this section on locally conformally K\" ahler manifolds that are Calabi--Yau and find sufficient conditions to ensure they actually are Vaisman.
We first note that the Calabi--Yau class is disjoint from the Einstein--Weyl class, defined as locally conformally K\" ahler manifolds for which the universal cover $(\widetilde{M}, J, \tilde{g})$ is Ricci flat. 
This latter family has been shown to be Vaisman in \cite{gau} and could be regarded as a local metric analogue of the Calabi--Yau notion. 
Both these classes have vanishing first Chern class in $H^2_{dR}(X, \mathbb{R})$, but are nevertheless complementary, when we pass to the finer Bott--Chern cohomology. Indeed, if $(M, J, g)$ is a locally conformally K\" ahler manifold with Chern--Ricci flat metric $g$, then the universal cover $(\widetilde{M}, J, \tilde{g})$ satisfies $\mathrm{Ric}_{\tilde{g}}=\pi^*dJ\theta$, but $dJ\theta$ cannot be $dd^c$-exact, hence not vanishing in the Bott--Chern cohomology. 

The first example of Vaisman Calabi--Yau manifold is given in complex dimension 2 by the primary Kodaira surface, that has holomorphically trivial canonical bundle. Moreover, this surface is endowed with a holomorphic symplectic form, which is, in general, one of the instances when a complex manifold has holomorphically trivial canonical bundle. However, in higher dimensions, a Vaisman Calabi--Yau manifold cannot admit such structures, thanks to \cite{ao23}.  

Inspired by Gauduchon's result in the Einstein--Weyl setting, we prove the following result, which shows that, under certain conditions on the Chern--Ricci form, an lcK metric is Vaisman.  

\begin{theo*}[Theorem \ref{thm A}]
   Let $(M, J)$ be a compact manifold admitting a locally conformally K\" ahler metric $\omega$ such that $Ric(\omega)=tdJ\theta$, where $\theta$ is the Lee form of $\omega$ and $t \leq 0$. If $\omega$ is Gauduchon, then $\omega$ is Vaisman. In particular, an lcK metric which is Gauduchon and Chern--Ricci flat is Vaisman.
\end{theo*}

\begin{proof} 
    The proof rests heavily on the fact that we have an explicit formula for $Ric^{W}$ and implicitly for $Ric^{LC}$. Indeed, by \cite[Theorem 1.159]{besse}, the relation between $Ric^{W}$ and $Ric^{LC}$ is the following:
       \begin{equation}\label{eq: Ric W-LC}
            Ric^W = Ric^{LC} +\frac{1}{2}\left(-d^*\theta +(1-n)\|\theta\|^2\right)g + (n-1)\Na\theta  + \frac{1}{2}(n-1)\theta\otimes\theta,
        \end{equation}
    as a straightforward consequence of $\nabla^{W}$ being locally the Levi-Civita connection of the conformal K\" ahler metric. 
    Moreover, since $Ric(\omega)=tdJ\theta$, by \eqref{Weyl-form}, we get $Ric^{W, J}=\left(\frac{n}{2}+t\right)dJ\theta$ and therefore, $Ric^{W}=\left(\frac{n}{2}+t\right)dJ\theta(\cdot, J\cdot)$.  
   Now, by taking the Lie derivative $\mathcal{L}_{\theta^{\sharp}}$ in the equation $\omega=g(J \cdot, \cdot)$, one can easily prove that
    \begin{equation}\label{dJtheta}
        dJ\theta (\cdot, J \cdot)= 2 (\nabla \theta)^{1, 1} - \|\theta\|^2g + \theta\otimes \theta + J\theta \otimes J\theta,
    \end{equation}
    where $(\nabla \theta)^{1, 1}:=\frac{1}{2}\left(\nabla \theta+ \nabla \theta (J\cdot, J\cdot)\right)$ is the $J$-invariant part of $\nabla \theta$.
    Therefore, combining with \eqref{eq: Ric W-LC} we get 
    \begin{equation}\label{Ricci}
        Ric^{LC}=(n+2t) (\nabla\theta)^{1, 1}-\left(\left(t+\frac{1}{2}\right)\|\theta\|^2-\frac{d^*\theta}{2}\right)g-(n-1)\nabla \theta+\frac{2t+1}{2}\theta \otimes \theta+\left(\frac{n}{2}+t\right) J\theta \otimes J\theta.
    \end{equation}


    
    We now consider the Bianchi identity, which can be written as 
    \begin{equation*}\label{Bianchi}
        \delta Ric^{LC}+\frac{1}{2}ds^{LC}=0,
    \end{equation*}
    where $\delta$ is the $L^2$-adjoint of the operator symmetrizing $\nabla$,  $\nabla^{sym}: \Omega^1(M) \rightarrow Sym^2$.
    By taking the $L^2$ product with $\theta$ and taking into account that $d^*\theta=0$, we obtain:
    $$\int_M \left< Ric^{LC},\nabla\theta \right>_g \mathrm{vol}(\omega)=0$$
    Hence, substituting $Ric^{LC}$ with \eqref{Ricci} and using that $\left< g,\nabla\theta \right>_g= -d^*\theta=0$ we get:

    \begin{equation*}
        \int_M\left((n+2t)||(\nabla \theta)^{1,1}||^2-(n-1)||\nabla \theta||^{2}+\left(\frac{n}{2}+t\right) \nabla \theta (J\theta^{\sharp}, J\theta^{\sharp})+\frac{2t+1}{2}\nabla \theta (\theta^{\sharp}, \theta^{\sharp})\right) \mathrm{vol}(\omega)=0.
    \end{equation*}
However, $\nabla \theta (\theta^{\sharp}, \theta^{\sharp})=\frac{1}{2}\theta^{\sharp}(\|\theta\|^2)$, as a consequence of $\nabla_{\theta^{\sharp}}g(\theta^{\sharp}, \theta^{\sharp})=0$, whence 
$$\int_{M}\nabla \theta (\theta^{\sharp}, \theta^{\sharp})\,\mathrm{vol}(\omega)=\frac{1}{2}\int_M\|\theta\|^2d^*\theta\,\mathrm{vol}(\omega)=0.$$
Therefore, we are left with:
\begin{equation}\label{simplified}
    \int_M\left((n+2t)||(\nabla \theta)^{1,1}||^2-(n-1)||\nabla \theta||^{2}+\left(\frac{n}{2}+t\right) \nabla \theta (J\theta^{\sharp}, J\theta^{\sharp})\right) \mathrm{vol}(\omega)=0.
\end{equation}

We shall now address the term $\nabla\theta(J\theta^\sharp,J\theta^\sharp)$. To do so we use the Weitzenb\"ock formula for one-forms $\eta$, that reads as
    \begin{equation*}
        \Delta \eta = \Na^*\Na\eta+ \iota_{\eta^\sharp}Ric^{LC},
    \end{equation*}
    where $\Delta$ is the Laplacian $dd^*+d^*d$, and $\Na^*$ is the $L^2$-dual of $\Na$, see e.g. \cite[Corollary 21]{peter}. 
    Since $\omega$ is Gauduchon and lcK, $\Delta\theta=0$. Hence, specializing this formula for the Lee form $\theta$ and integrating against it we obtain
    \begin{equation}\label{bochner}
        0 = \int_M\|\Na\theta\|^2 \mathrm{vol}(\omega) + \int_M Ric^{LC}(\theta^\sharp,\theta^\sharp) \mathrm{vol}(\omega).
    \end{equation}
However, by \eqref{Ricci}, we get 
\begin{equation*}
    Ric^{LC}(\theta^{\sharp}, \theta^{\sharp})=\left(\frac{n}{2}+t \right)(\nabla \theta(\theta^{\sharp}, \theta^{\sharp})+\nabla \theta(J\theta^{\sharp}, J\theta^{\sharp}))-(n-1)\nabla \theta (\theta^{\sharp}, \theta^{\sharp}),
\end{equation*}
and replacing it into \eqref{bochner}, we have 
\begin{equation}\label{Jtheta}
    -\int_M||\nabla \theta||^2 \mathrm{vol}(\omega)=\left(\frac{n}{2}+t \right)\int_M\nabla \theta(J\theta^{\sharp}, J\theta^{\sharp})\mathrm{vol}(\omega)
\end{equation}
Finally, combined with \eqref{simplified}, we obtain
\begin{equation}\label{eqt}
    \int_M\left((n+2t)||(\nabla \theta)^{1,1}||^2-n||\nabla \theta||^{2}\right) \mathrm{vol}(\omega)=0.
\end{equation}
Since $t \leq 0$, $(n+2t)||(\nabla \theta)^{1,1}||^2-n||\nabla \theta||^{2} \leq 0$. We distinguish two cases. If $t<0$, then equality can be attained if and only if $\nabla \theta=0$, which means $\omega$ is Vaisman. If $t=0$, equality is attained if and only if $\nabla\theta=(\nabla \theta)^{1,1}$. This further implies $\nabla \theta (J\theta^{\sharp}, J\theta^{\sharp})=\nabla\theta(\theta^{\sharp}, \theta^{\sharp})$, therefore by \eqref{Jtheta}, $\nabla\theta=0$ again. So $\omega$ is Vaisman in both cases.
\end{proof}

\begin{rmk}
    The Gauduchon condition in Theorem \ref{thm A} cannot be dropped. As a matter of fact, one can construct lcK metrics that are Chern--Ricci flat and are not Gauduchon as we do in the following.
    However, we shall remark that the exmples we provide are either on Vaisman manifolds or on K\"ahler manifolds, in agreement with our conjecture. 
\end{rmk}

\begin{ex*}
    We use that a Vaisman metric on a compact manifold is prescribed by the cohomology class of its Lee form and by its volume form, as stated in \cite{ov24}.
    Thus, take $(M, J)$ a compact Vaisman Calabi--Yau manifold and fix a Lee class $\tau \in H^1_{dR}(M)$. By \cite{i} there exists a unique Vaisman metric $g$ with Lee class $\tau$ that is Chern--Ricci flat. Consider now a smooth function $f$ on $M$ that is invariant with respect to the Lee group generated by $\theta^{\sharp}$ and $J\theta^{\sharp}$. By \cite{ov24}, there exists a unique Vaisman metric $g_1$ with Lee class $\tau$ and volume form $\mathrm{vol}_{g_1}=f\mathrm{vol}_g$. Then $Ric(g_1)=-\mathrm{i}\partial\overline{\partial}f$ and $Ric(e^{-\frac{f}{n}}g_1)=0$. Since $g_1$ is Vaisman and $f$ is not constant, $e^{-\frac{f}{n}}g_1$ is a non-Gauduchon lcK metric, that is Chern--Ricci flat. 
    \hfill$\blacksquare$
\end{ex*}
\begin{ex*}
    Starting from a K\"ahler Calabi--Yau metric $\omega$, we consider a different element in the K\"ahler class $\omega_\varphi := \omega +\im\p\bar\p\varphi$. This will satisfy $Ric(\omega_\varphi)=\im\p\bar\p f$ for some smooth function $f$. Then $e^{\frac{f}{n}}\omega_\varphi$ is globally conformally K\"ahler and Chern--Ricci flat.
    \hfill$\blacksquare$    
\end{ex*}

As a direct consequence of \eqref{eqt}, we can extrapolate from the above proof the following result.
\begin{cor}
    Let $(M, J)$ be a compact manifold admitting a locally conformally metric $\omega$ such that $Ric(\omega)=tdJ\theta$, for $t \in \mathbb{R}$. If $(M, J)$ is not Vaisman, then $t=\frac{n}{2}\left(\frac{||\nabla \theta||^2}{||\nabla \theta^{1, 1}||^2}-1\right)>0$.
\end{cor}

Finally, we notice that the Oeljeklaus--Toma manifolds (see Section \ref{s: OT}) provide examples of compact manifolds admitting an lcK metric $\omega$ such that $Ric(\omega)=tdJ\theta$, for $t>0$, which is not Vaisman.

\section{Calabi--Yau locally conformally Kähler solvmanifolds}\label{sec: solv}
The simplest examples of K\"ahler Calabi--Yau manifolds are complex tori.
Analogously, the simplest examples of compact non-K\"ahler Vaisman manifolds which are Calabi--Yau are given by principal toric bundles over them. 
These are the {\em Kodaira manifolds} defined in \cite{gmpp}.
They are the high-dimensional analogous of {\em Kodaira surfaces} and have a nilmanifold structure given by $\mathbb{Z} \times H_{\mathbb{Z}} \backslash \mathbb{R} \times H_{2n+1}$. Here $H_{2n+1}$ denotes the Heisenberg group
\[
H_{2n+1} = \left\{ 
\begin{pmatrix}
1 & x & z \\
0 & I_n & y^\top \\
0 & 0 & 1
\end{pmatrix} 
: x, y \in \mathbb{R}^n, \, z \in \mathbb{R} 
\right\}
\]
and $H_{\mathbb{Z}}$ is $\mathbb{Z}^{2n+1}$ seen as a subgroup of $H_{2n+1}$.
We shall prove that, up to finite cover, these are all the solvmanifolds with left-invariant Vaisman structure.
To do so, one builds on the characterization by Andrada--Origlia of unimodular Vaisman Lie algebras as double extensions of a K\"ahler flat Lie algebras.

\begin{theo}[Theorem 3.9 in \cite{ao}]\label{thm: AO}
    Let $\mathfrak{g}$ be a unimodular solvable Lie algebra endowed with a Vaisman metric $\omega$ with Lee form $\theta$. Then $\mathfrak{g}=\mathbb{R}\theta \ltimes_{D} (\mathbb{R}J\theta \oplus \mathfrak{l})$, where $\mathfrak{l}$ is a K\" ahler Lie algebra, $D=ad_{\theta}|_{\mathrm{Ker}\, \theta}$ and the structure on $\mathbb{R}J\theta \oplus \mathfrak{l}$ is given by $$[x, y]=-dJ\theta(x, y)J\theta+[x, y]_{\mathfrak{l}}.$$ 
    Moreover, $-dJ\theta$ defines a flat K\" ahler metric on $\mathfrak{l}$.
\end{theo}

We can thus prove the main result of the section.
    
\begin{theo*}[Theorem \ref{thm C}]\label{structure} 
    Let $(\Gamma \backslash G, J)$ be a compact solvmanifold with left-invariant complex structure admitting a left-invariant lcK metric. The following are equivalent:
    \begin{enumerate}
      \item $(\Gamma \backslash G, J)$ is Calabi--Yau.
      \item $(\Gamma \backslash G, J)$ admits a left-invariant Vaisman metric.
    \end{enumerate}
    Moreover, such a solvmanifold is diffeomorphic to a finite quotient of the Kodaira manifold.
\end{theo*}
     
\begin{proof} 
    We start by proving $(1) \Rightarrow (2)$. 
    Let $\omega$ be a left-invariant lcK metric with corresponding Lee form $\theta$. 
    Then $Ric(\omega)$ is a left-invariant two-form and moreover, $\theta$ is also left-invariant, implying $\omega$ is Gauduchon. 
    If $(\Gamma \backslash G,J)$ is Calabi--Yau, then $Ric(\omega)=\mathrm{i}\partial\overline{\partial}f$, but since this form is left-invariant, it has to vanish. 
    Therefore, by Theorem \ref{thm A}, $\omega$ is Vaisman. 
    
    Conversely, if $\omega$ is a left-invariant Vaisman metric with Lee form $\theta$ of norm 1, since $\theta^{\sharp}$ and $J\theta^{\sharp}$ form the canonical foliation and $\omega=\theta \wedge J\theta - dJ\theta$, we have $Ric(\omega)=Ric(-dJ\theta)$, with $-dJ\theta$ being the transverse K\" ahler structure. 
    By Theorem \ref{thm: AO}, $-dJ\theta$ is a flat Kähler metric, hence we derive $Ric(\omega)=0$, implying $(\Gamma \backslash G, J)$ is Calabi--Yau.

    We now argue that these manifolds are diffeomorphic to finite quotient of the Kodaira manifold.
    It was proven in \cite[Corollary 5.10]{i} that the fundamental group of a compact Vaisman Calabi--Yau manifold contains a finite index normal subgroup which is either odd-dimensional abelian or $\mathbb{Z}\times H_{\mathbb{Z}}$.
    If the fundamental group contains a finite index normal abelian subgroup, $\Gamma \backslash G$ is finitely covered by a solvmanifold $X$ with abelian fundamental group. Then, thanks to \cite{mos} $X$ is necessarily diffeomorphic to odd-dimensional a torus, since it is compact and has a solvmanifold structure itself, which is impossible being complex.
    This implies, therefore, that the fundamental group has a finite indexed normal subgroup isomorphic to $\mathbb{Z} \times H_{\mathbb{Z}}$, therefore $\Gamma \backslash G$ is diffeormorphic to a finite quotient of the Kodaira manifold.
\end{proof}

\begin{rmk}
    As it is evident from the proof, the first implication $(1) \Rightarrow (2)$ holds for any homogeneous manifolds as we are only exploiting the fact that the Hermitian structure is invariant. The solvmanifold structure plays a role only in the reverse implication.
\end{rmk}

\begin{rmk}
    In \cite{ao} the authors use oscillator groups to construct compact Vaisman solvmanifolds which are not isomorphic to Kodaira manifolds. 
    However, from their construction it is clear that these are finitely covered by them. We also remark that Vaisman solvable Lie algebras need not to be nilpotent, however, we proved that their compact quotients are isomorphic to a nilmanifold up to finite cover. More precisely, for a solvmanifold $\Gamma\backslash G$ with left-invariant Vaisman structure the Lie algebra of $G$ is in general as in Theorem \ref{thm: AO} but up to finite cover $\Gamma\backslash G \cong \mathbb{Z} \times H_{\mathbb{Z}} \backslash \mathbb{R} \times H_{2n+1}$.
\end{rmk}

\begin{prop}
    Let $(\Gamma \backslash G, J)$ be a compact solvmanifold with left-invariant complex structure. Then left-invariant lcK metrics with potential are Vaisman.
\end{prop}

\begin{proof}
    Let us set $M:=\Gamma \backslash G$, and let $\omega$ be a left-invariant lcK metric with potential on $(M, J)$.
    From the definition of lcK metric with potential we know that there exists a smooth function $f\in C^\infty(M)$ such that $$\omega = d_\theta Jd_\theta f.$$
    Since $G$ is a simply connected Lie group admitting a co-compact discrete subgroup, it is unimodular \cite[Lemma 6.2]{mil} and hence it admits a bi-invariant volume form $\nu$. 
    Since both $\omega$ and $J$ are left-invariant, then also $\theta$ is.
    Thus, by an averaging procedure using the bi-invariant volume form, we can assume $f$ left-invariant, i.e. constant.
    Let us provide more details on this.

    First of all, we suppose that $\nu$ has unitary volume. Given any $k$-form $\alpha\in\A^k(M)$ on the solvmanifold $M$, we define a $k$-form $\alpha_{\nu}\in\A^k(\g)$ on the Lie algebra of $G$ by
    $$
    \alpha_{\nu}(X_1,\ldots,X_k)=\int_{p\in M} \alpha_p
    (X_1\!\mid_p,\ldots,X_k\!\mid_p) \, \nu\ , \quad\quad \mbox{for}\
    X_1,\ldots,X_k\in \g,
    $$
    where $X_j\!\mid_p$ is the value at the point $p\in M$ of the
    projection on $M$ of the left-invariant vector field $X_j$ on the
    Lie group $G$.
    Obviously, $\alpha_{\nu}=\alpha$ for any tensor field $\alpha$ coming from a
    left-invariant one. 
    Moreover, in~\cite{bel00} it is proved that for any $k$-form $\alpha$ on $M$,
    $(d\alpha)_{\nu}=d\alpha_{\nu}$, while, a simple calculation shows that
    $(\alpha_\nu\wedge\beta)_\nu=\alpha_\nu\wedge\beta_\nu$, for any
    $\alpha \in \A^k(M)$ and $\beta \in \A^l(M)$.
    Therefore, the averaging commutes with $d_\theta=d-\theta_\nu\wedge\cdot$.
    Finally, since $J$ is invariant, we also have $(J\alpha)_\nu = J\alpha_\nu$ for any $\alpha\in\A^k(M)$. 
    Putting these together we get
    $$\omega=\omega_\nu=\left(d_\theta Jd_\theta f\right)_\nu = d_\theta Jd_\theta f_\nu.$$

    It follows that $\omega$ satisfies $\omega= \tfrac{1}{\|\theta\|^2}(\theta\wedge J\theta - dJ\theta)$, with $\|\theta\|^2$ constant, which can be taken equal to $1$.
    Then it has to be Vaisman due to \cite[Proposition 3.2.2]{i18}.
\end{proof}

\begin{cor}
    Let $(M, J)$ be a solvmanifold with left-invariant complex structure and a left-invariant lcK metric. If $\mathrm{dim}_{\mathbb{R}} H^{1, 1}_{BC}=1$, then $c_1^{BC}(M)$ is either 0, which is the Vaisman case, or $\left[\frac{n}{2}\left(\frac{||\nabla \theta||^2}{||\nabla \theta^{1, 1}||^2}-1\right)dJ\theta\right]$, which is the non-Vaisman case.
\end{cor}

\section{Relation with Sasaki geometry}\label{s: sasaki}
Sasaki metrics are geometric structures related to  Vaisman manifolds via their strong connection to K\" ahler manifolds. By definition, a $(2n-1)$-dimensional Riemannian manifold $(S, g_S)$ is called \textit{Sasaki} if the cone $S \times \mathbb{R}_+$ admits a complex structure for which the metric $dr^2+r^2g_S$ is K\"ahler. 
By the proof of the structure theorem in \cite{ov03}, any Vaisman manifold $(M, J, g)$ admits a $\mathbb{Z}$-cover $\pi: S \times \mathbb{R}_+ \rightarrow M$, where $\mathbb{Z}$ acts by $(x, r) \mapsto (\phi(x), \lambda r)$ and $\phi$ is an automorphism of $S$ and $\lambda \neq 1$. Moreover, if $g$ is of rank 1, then $dr^2+r^2g_S=e^{-f}\pi^*g$, where $f$ satisfies $df=\pi^*\theta$ for the Lee form $\theta$. Since the cone metric is actually $\frac{\mathrm{i}}{2}\partial\overline{\partial}r^2$, we derive $\pi^*\theta=-2d\, \mathrm{log}\, r$. 

We can make a connection between Vaisman metrics of rank 1 which satisfy the Einstein-type condition $Ric(\omega)=tdJ\theta$ and the {\em Sasaki $\eta$-\textit{Einstein} metrics}. By definition, these are given by the following condition on the Ricci tensor:
\begin{equation*}
    \mathrm{Ric}_{g_S}=\alpha g_S+\beta \eta \otimes \eta,
\end{equation*}
where $\eta:=d^c\, \mathrm{log}\, r$ and $\alpha, \beta \in \mathbb{R}$, with $\alpha+\beta=2n-2$.

\begin{prop}
    Let $M$ be a manifold of $\mathrm{dim}_{\mathbb{C}}=n$. Any Vaisman metric of rank 1 with $Ric(\omega) = tdJ\theta$, for $t\in\R$, on $M$ corresponds to a Sasaki $\eta$-Einstein metric satisfying:
    \begin{equation}\label{ricci-sasaki}
        \mathrm{Ric}_{g_S}=-(2+4t) g_S + (2n+4t)\, \eta \otimes \eta
    \end{equation}
\end{prop}
\begin{proof}
    Consider the $\mathbb{Z}$-cover $\pi:S \times \mathbb{R}_+ \rightarrow M$. Since $g$ is of rank 1, $\tilde{g}:=e^{-f}\pi^*g=dr^2+r^2g_S$. This is a warped product metric and, therefore, its Ricci tensor is given by \cite[Proposition 9.106]{besse}: 
    \begin{equation}\label{eq: wrap}
        \mathrm{Ric}_{g_S}=Ric(\tilde{\omega})\circ \Phi+(2n-2)g_S,
    \end{equation}
    where $\Phi$ is the endomorphism of $TS$ induced by the complex structure of the cone $S \times \mathbb{R}_{+}$. We denote here by the composition of a two-form $\alpha$ with the endomorphism $\Phi$ the tensor obtained by $\alpha \circ \Phi (X, Y):=\alpha(X, \Phi Y)$.\\
    By the hypothesis on the Chern--Ricci form and the definition of $\eta$ we get
    \begin{align*}
        Ric(\tilde{\omega}) &= \pi^*Ric(\omega) + \frac{n}{2}dd^cf
            = t\pi^*(dJ\theta) + \frac{n}{2}dJdf
            = \left(t+\frac{n}{2}\right)dJ\pi^*\theta
            =-(2t+n)d\eta.
    \end{align*}
    Combined with \eqref{eq: wrap} and $g_S=\eta \otimes \eta + \frac{d\eta}{2} \circ \Phi$, one retrieves \eqref{ricci-sasaki}. 
\end{proof}
Conversely, any Sasaki $\eta$-Einstein metric on $S$ satisfying \eqref{ricci-sasaki} gives a Vaisman metrics on the manifolds obtained by quotienting $S \times \mathbb{R}_+$ to $(x, r) \mapsto (\phi(x), \lambda r)$ which satisfies $Ric(\omega) = tdJ\theta$.


\end{document}